\documentclass[12pt,oneside,a4paper]{article}
\usepackage{amsmath}
\usepackage{amsthm}
\newtheorem{defn}{Definition}[section]
\newtheorem{thm}{Theorem}[section]
\newtheorem{lem}[thm]{Lemma}
\newtheorem{prop}[thm]{Proposition}

\usepackage{graphicx}
\usepackage{sidecap}
\graphicspath{{images/}}
\begin{document}
	\title{Improving Sharir and Welzl's bound on crossing-free matchings through solving a stronger recurrence}
	\author{Chenchao You}
	\date{}
	\maketitle
	
	\section*{Abstract}
		Sharir and Welzl [1] derived a bound on crossing-free matchings primarily based on solving a recurrence based on the size of the matchings. We show that the recurrence given in Lemma 2.3 in Sharir and Welzl can be improve to \((2n-6s)\textbf{Ma}_{m}(P)\leq\frac{68}{3}(s+2)\textbf{Ma}_{m-1}(P)\) and \((3n-7s)\textbf{Ma}_{m}(P)\leq44.5(s+2)\textbf{Ma}_{m-1}(P)\), thereby improving the upper bound for crossing-free matchings.

	\section{Introduction} 
		A \textit{geometric graph} on a set of points \(P\) on a plane has \(P\) as its vertex set and edges being straight lines. We are interested in the number of various crossing-free geometric graphs such as matchings, spanning cycles and triangulations. Sharir and Welzl [1] derived a upper bound of \(O(10.05^n)\) for matchings in 2006 by \textit{vertical decomposition} and \textit{trapezoidization}. A recurrence is derived on the number of ways to insert and remove edges incident to lower rank. We improve Sharir and Welzl's results by showing a lower bound for the number of ways to insert edges incident to low rank.
		
		Let \(P\) be a set of $n$ points in the plane in general position (no three points on a line and no two points in a vertical line). A crossing-free matching is a graph with pair-wise non-adjacent edges such that no two edges (straight lines) cross each other. We denote \(M\) as one particular configuration of crossing-free matchings. The size of \(M\) is the number of edges in the configuration, denoted as \(|M| = m\). \textit{Trapezoidization} is the process of partitioning the plane into trapezoids; vertical lines are drawn through each vertex in \(P\), blocked by edges in crossing-free matchings to serve as parallel bases. We say a point is \textit{vertically} visible from an edge if the vertical line through the point is segmented by the edge. We first define the \textit{rank} of each vertex in \(P\).
		\begin{defn}
			Given a point set \(P\), a point \(p\in P\), and a matching \(M\) of \(P\), \(d(p)\) which is the rank of \(p\) in \(M\) is zero if \(p\) is not matched in \(M\). If \(p\) is matched as a left (resp., right) endpoint, \(d(p)\) is equal to the number of isolated vertices and left (resp., right) endpoints vertically visible from the interior of the edge via which \(p\) connects.
		\end{defn}
		Let \(M\) be a crossing-free matching of a point set \(P\) in which the point \(p\in P\) is isolated. We denote by \(h_i(p,M)\) the number of ways to insert an edge with \(p\) as an endpoint, such the rank of \(p\) becomes \(i\) after insertion. Let \(H_k(p,M) = \sum_{i=0}^{k}(k-i)h_i(p,M)\). Sharir and Welzl [2] considered \(k=4\) and \(k=5\) for \(H_k\) and proved the following lemma:
		\begin{lem}
			Denote by \(h_i(p,M)\) the number of ways to insert an edge with \(p\) as an endpoint in a crossing-free matching \(M\) of given point set \(P\), such that the rank of \(p\) becomes \(i\). \(\sum_{i=0}^{4}(4-i)h_i(p,M)\leq 24\) and \(\sum_{i=0}^{5}(5-i)h_i(p,M)\leq 48\) for any given set of \((p,M)\).
		\end{lem}
		We denote by \(\textbf{Ma}_{m}(P)\) the number of crossing-free matchings \(M\) with \(m\) edges that can be embedded over the point set \(P\) of size \(n\). Let \(s\) denote the number of isolated vertices in \(M\) and notice that \(s = n - 2m\). Sharir and Welzl also showed the following proposition:
		\begin{prop}
			Denote by \(v_i(M)\) the number of points with rank \(i\) in a crossing-free matching \(M\) of a given point set \(P\). \(\sum_{i=0}^{4}(4-i)v_i(M)\geq 2n-6s\) and \(\sum_{i=0}^{5}(5-i)v_i(M)\geq 3n-7s\).
		\end{prop}
		We can derive a matching of size \(m\) by inserting an edge into a matching of size \(m-1\). Conversely, a matching of size \(m-1\) can be obtained by removing an edge from a matching of size \(m\). We say that an edge is incident to low rank if the endpoints of the edge has low rank. Sharir and Welzl argued that if we restrict to only insert and remove an edge that is incident to low rank, we can combine Lemma 1.1 and Proposition 1.2 to obtain a recurrence on \(\textbf{Ma}_{m}(P)\) as shown in the following lemma:
		\begin{lem}
		Let \(\textbf{Ma}_{m}(P)\) the number of crossing-free matchings \(M\) with \(m\) edges that can be embedded over the point set \(P\) of size \(n\). We have the following inequalities:
		\[(2n - 6s)\textbf{Ma}_{m}(P)\leq \sum_{p,M \atop |M|=m-1}\sum_{i=0}^{4}(4-i)h_i(p,M)\leq 24(s+2)\textbf{Ma}_{m-1}(P),\]
		\[(3n - 7s)\textbf{Ma}_{m}(P)\leq \sum_{p,M \atop |M|=m-1}\sum_{i=0}^{5}(5-i)h_i(p,M)\leq 48(s+2)\textbf{Ma}_{m-1}(P),\]
		\begin{equation*}
			\textbf{Ma}_{m}(P)\leq
			\begin{cases}
			\frac{24(s+2)}{2n-6s}\textbf{Ma}_{m-1}(P)\\
			\frac{48(s+2)}{3n-7s}\textbf{Ma}_{m-1}(P)\\
			\end{cases}.
		\end{equation*}
		\end{lem}
		Sharir and Welzl solved the recurrence on \(\textbf{Ma}_{m}(P)\) above to obtain the upper bound of \(O(10.05^n)\) for matchings. We improve on their results by deriving the following stronger upper bounds and thus obtaining a more powerful recurrence:
		\begin{equation}
			\sum_{p,M \atop |M|=m-1}\sum_{i=0}^{4}(4-i)h_i(p,M)\leq \frac{68}{3}(s+2)\textbf{Ma}_{m-1}(P)
		\end{equation}
		\begin{equation}
		\sum_{p,M \atop |M|=m-1}\sum_{i=0}^{5}(5-i)h_i(p,M)\leq 44.5(s+2)\textbf{Ma}_{m-1}(P).
		\end{equation}
		To obtain the improved upper bounds above, we look into the specific values of \(H_k(p,M)\) over the set of \((p,M)\). In particular, we denote each pair of \(p\) and \(M\) a ving (Vertex in Graph) and show that there are many vings \((p,M)\) such that \(H_k(p,M)\) is strictly less than the bounds given in Lemma 1.1. This lowers the sum of \(\sum_{i=0}^{k}(k-i)h_i(p,M)\) over the set of all \((p,M)\) and therefore improves the bounds in Lemma 1.3.

	\section{Setup and Constellation}
		In this section, we look at the different values of \(H_k(p,M)\) and the specific configurations that give the value of \(H_k(p,M)\). Notice that a point \(p\) can either connect to the left as a right endpoint, or it can connect to the right as a left endpoint. Thus, it is reasonable to split \(H_k(p,M)\) into \(H_k(p,M) = L_k(p,M) + R_k(p,M)\), where \(L_k(p,M)\) counts only the number of ways to insert an edge that emanates from \(p\) to the left , and \(R(p,M)\) counts only to the right. It is evident that \(L_k(p,M)\) and \(R_k(p,M)\) are symmetric, thus it suffices to consider points and edges only on one side to enumerate all possible configurations. We begin by a few definitions to setup the argument.
		\begin{defn}
			We say a point \(p\) sees a point \(q\) if we can move from \(p\) through vertical walls to adjacent trapezoids until we reach a vertical wall determined by the point \(q\). Notice that the edge \((p,q)\) may intersect other edges.
		\end{defn}
		
		\begin{defn}
			Under vertical decomposition of a matching \(M\) of a point set \(P\) and for any \(p\in P\), a left (resp., right) bifurcation point of \(p\) is the first point that an isolated point \(p \in P\) sees to the left (resp., right).
		\end{defn}
		
		\begin{prop}
			For any matching \(M\), if the left bifurcation point of \(p\) is isolated, then \(L_3(p,M)\leq 10\) and \(L_4(p,M)\leq 17\). By symmetry the bounds also apply to \(R_3(p,M)\) and \(R_4(p,M)\).
		\end{prop}
			\begin{proof}
				If the left bifurcation point \(q\) of \(p\) is isolated, suppose that if we insert an edge \(e\) emanating from \(p\) to the left then \(p\) gets rank of \(k\geq 1\). Then after removing \(q\), \(e\) gives \(p\) a rank of \(k-1\). Notice that \(l_0(p,M)\leq 1\) since we can only connect \(p\) to \(q\) to have the \(d(p)=0\). Sharir and Welzl proved that \(\sum_{i=0}^{3}(3-i)l_i(p,M)\leq 6\) and we obtain the recurrence
				\[\sum_{i=0}^{4}(4-i)l_i(p,M)\leq \sum_{i=0}^{3}(3-i)l_i(p,M) + 4l_0(p,M)\leq \sum_{i=0}^{3}(3-i)l_i(p,M) + 4\leq 10 .\]
				Similarly,
				\[\sum_{i=0}^{5}(5-i)l_i(p,M)\leq \sum_{i=0}^{4}(4-i)l_i(p,M) + 5l_0(p,M)\leq \sum_{i=0}^{4}(4-i)l_i(p,M) + 5\leq 17.\]
			\end{proof}
		
		\begin{lem}
				Consider a matching \(M\) of a point set \(P\) and points \(a,b\in P\) such that \(a\) and \(b\) have the same left (resp., right) bifurcation point \(p\). Without loss of generality consider \(a\) above \(b\) (having a larger \(y\)-coordinate). There exist at least one edge \(e\) directly below \(a\) and at least one edge \(e'\) directly above \(b\) (Notice that \(e\) and \(e'\) may coincide). The choice of the left (resp., right) endpoint of \(e\) and \(e'\) are unique in \(P\) given the \(M\) and the position of point \(a\) and point \(b\).
		\end{lem}
			\begin{proof}
				Let the bifurcation point to the left be \(q\) and to the right be \(q'\). Let two points \(a\) and \(b\) be described above. Without loss of generality let \(a\) be above \(b\). Consider the edge directly below \(a\). Such edge exists because \(a\) and \(b\) must be separated by at least one edge, or either \(a\) sees \(b\) before seeing bifurcation point or \(b\) sees \(a\) before seeing bifurcation point.
				
				We remove this edge and prove that the two endpoints of this edge are unique. Denote the left endpoint of this edge \(u\) and right endpoint of this edge \(u'\). Consider going from point \(a\) to the left through trapezoids until reaching bifurcation point \(q\). We immediately see that there are no isolated vertices left of \(u\) before bifurcation point \(q\).
				
				If we look at the configuration after removing edge \(uu'\), we see a set \(U\) of isolated vertices when going from vertex \(a\) to the left until bifurcation point \(q\). We claim that \(u\in U\) is above the rest of such points in the set. This is true because removing edge \(uu'\) which is below \(a\) do not effect trapezoids above. If there are such isolated vertices above \(u\in U\), \(a\) will see this vertex before seeing bifurcation point \(q\) before removing \(uu'\), which is a contradiction.
				
				The same argument applies to the right side for the right endpoint \(u'\). Thus it is safe to uniquely determine the two endpoints and reconstruct the edge after removing it. 
			\end{proof}
			
		It is then natural to consider cases where no two points see the same bifurcation point both to the left and to the right.
		
		\begin{lem}
			For a set of \(k\) isolated vertices that are separated by edges directly below and above, with no two vertices seeing the same constellation and thus no edges ready for removal, we obtain at least \(k+1\) distinct bifurcation points (either to the left or to the right).
		\end{lem}
			\begin{proof}
				 We prove this by induction. The case \(k=2\) is trivial. Suppose it holds for \(k\) vertices. Without loss of generality consider adding a new vertex \(a\) above all \(k\) vertices. We obtain \(k+1\) points separated by edges below and above. Denote the edge directly below \(a\) edge \(uu'\). If no new bifurcation points are found for \(a\), \(a\) must see the same bifurcation point as some point \(b\) below \(a\) to the left, and the same bifurcation point as some point \(c\) below \(a\) to the right. This determines endpoints \(u\) and \(u'\) which forces edge \(uu'\) ready for removal, thus a contradiction. Thus the lower bound holds for all vertices of number \(k\).
			\end{proof}
		
		We are interested in the number of \textit{good points} in each constellation. Here we say a point \(p\) is good when it sees either isolated bifurcation points to the left or to the right, which forces \(H_3(p)\leq 22\) and \(H_4(p )\leq 41\). We establish the final set of lemmas before going into specific cases.
		
		\begin{defn}
			A constellation (to the left or to the right of a point) with bifurcation edge(i.e. bifurcation point is matched) is good if the bifurcation edge ``sees" all the isolated points in the constellation. ``Seeing" means all the isolated vertices lie between two endpoints of the bifurcation edge horizontally. A constellation is bad if it's not good.
		\end{defn}
		
		\begin{lem}
			If a point \(p\) sees a bad constellation, The bifurcation edge can be removed, leaving the bifurcation point and the other endpoint the the bifurcation edge unmatched. A edge emanating from the other endpoint can be uniquely reconstructed. 
		\end{lem}
			\begin{proof}
				We consider removing the bifurcation edge \(qq'\) where \(q\) is the bifurcation point. We claim that if the constellation is bad, \(q'\) is uniquely determined. Indeed, if a constellation is bad, the bifurcation edge \(qq'\) can not see all the isolated vertices in the constellation. Thus, after removing edge \(qq'\), the point \(q'\) can be uniquely identified as the \(i^{th}\) closest isolated vertex to bifurcation point \(q\) when going through trapezoids. Thus, the similar argument in Lemma 2.2 can be applied to conclude that edge \(qq'\) is ready for removal and reconstruction.
			\end{proof}
		
		\begin{lem}
			We say a vertex is charged by a constellation if the point contributes to the rank of the constellation. We claim that each good point can be charged by at most two different good constellations that realize \(L_3\) or \(R_3\) greater than 10 and \(L_4\) or \(R_4\) greater than 17. 
		\end{lem}
			\begin{proof}
				Notice that for constellations that realize the value of \(L_3\) or \(R_3\) greater than 10, and \(L_4\) or \(R_4\) greater than 17, the bifurcation points for each constellation must be matched and bifurcation edge exists.
				
				Since each good constellation has distinct bifurcation edges which see all the points, it is now obvious that the points can only be charged at most two times from above and below by two good constellations.
			\end{proof}
			
	\section{Obtaining a Correspondence}
		\subsection{\(H_3 = 4h_0 + 3h_1 + 2h_2 + h_3\)}
			By symmetry we first consider \(L_3\) that realizes the value of \(H_3\) to the left. By enumeration we see that three values of \((l_0, l_1, l_2, l_3)\) realize the value of \(L_3 = 12\), which are \((0, 2, 2, 2)\), \((0, 1, 3, 3)\), \((0, 0, 4, 4)\). The values that realize \(L_3 = 11\) are \((0, 0, 3, 5)\), \((0, 0, 4, 3)\), \((0, 1, 2, 4)\), \((0, 1, 3, 2)\).
			
			\begin{prop}
				The number of good points in each constellation that realizes \(L_3\) or \(R_3\) being 12 is 4. The number of good points in each constellation that realizes \(L_3\) or \(R_3\) being 11 is at least 2.
			\end{prop}
				\begin{proof}
					The result is immediately obtained by looking at each constellation.
				\end{proof}
				
			We first consider all constellations are good and no two points see the same constellation to the left and to the right.
			\begin{thm}
				 Consider any a set of \(k\) points \(\{p_1, p_2,..., p_k\}\), each having \(H_3(p_i, M) = x\) and separated by edges below and above, with all constellations being good and no two points seeing the same constellation. For \(x = 24\), we obtain at least \(2k + 2\) good points that uniquely corresponds to this setting. For \(x = 23\) we have at least \(k+1\) unique good points. 
			\end{thm}
				\begin{proof}
					By Lemma 2.3, we know that we obtain at least \(k+1\) distinct constellations and thus \(k+1\) distinct bifurcation points. Since all constellations realize \(L_3\geq 11\) and \(R_3\geq 11\), we obtain \(k+1\) distinct bifurcation edges.
					
					Since for each constellation that realizes \(L_3=12\) or \(R_3=12\), we have \(4\) good points which (by Lemma 2.4) can only be charged at most twice, we obtain \(\frac{4(k+1)}{2}=2k+2\) distinct good points for \(k+1\) such constellations. Similarly we obtain \(2\) good points for each constellation that realizes \(L_3=11\) or \(R_3=11\). Thus there are \(\frac{2(k+1)}{2}=k+1\) distinct good points.
				\end{proof}
			
			We have now obtained a unique correspondence between any \(k\) points that realizes \(H_3=24\) and \(2k+2\) good points with \(H_3\leq 22\). Also we obtain a unique correspondence between any \(k\) points that realizes \(H_3=23\) and \(k+1\) good points with \(H_3\leq 22\). Compare the average over all values of \(k\) and take the maximum, we obtain an average of \(\frac{68}{3}\)
			
			Thus \(\sum_{|M| = m - 1, P}4h_0 + 3h_1 + 2h_2 + h_3\leq \frac{68}{3}(s+2)|\textbf{Ma}_{m-1}(P)|\). This is true under the assumption that each point in \(P\) sees a good constellation and no two points in \(P\) see the same constellation to the left and right.
			
			We are then ready to look at the case when a point \(p\in P\) sees a bad constellation(either to the left or to the right).

			\begin{thm}
				Consider any point \(p\in P\) that sees the bad constellation to the left or to the right. We have a unique correspondence between a constellation $M$ with \(H_3(p,M) = 23\) or \(24\) and two or more distinct constellations \(M'\) with \(H_3(p,M') \leq 22\). That is, for each such constellation $M$ we can map $M$ to two or more constellations $M'$ and such map is surjective. 
			\end{thm}
				\begin{proof}
					Consider a point \(p\) such that \(H_3(a,M) = 23\) or \(24\). WLOG let \(p\) see bad constellation to the right. We claim that \(M'\) can be uniquely constructed from \(M\) by the following procedure: remove bifurcation edge \(qq'\) and insert an edge emanating from any good point \(a\) in the constellation to its closest vertex \(a'\) to the right. Thus the resulting \(M'\) leaves the bifurcation point of \(p\) to the right isolated, and \(H_3(p,M') \leq 22\).
					
					It is important to show that we can indeed connect \(a\) to \(a'\). Suppose there is an edge blocking the connection of \(a\) to \(a'\) and we denote the left endpoint of this edge point \(v\) and the right endpoint of this edge point \(v'\). Notice that \(p\) can be matched to both \(a\) and \(a'\). Thus we know that \(v\) must be within the triangle area \(\bigtriangleup paa'\). This is a direct contradiction to the assumption that no left endpoints are left of point \(p\) other than those shown in the constellation.
					
					We claim that there is a unique correspondence between each distinct \(M'\) and \(M\). This is indeed true because \(a\) is good point in the constellation and thus its closest vertex(to the right) is an isolated vertex shown in the constellation. Thus both \(a\) and \(a'\) can be uniquely identified. By Lemma 2.4, the bifurcation edge \(qq'\) can also be removed and uniquely reconstructed.
					
					Note that by Proposition 3.1, the number of good points that realizes \(R_3 = 11\) or \(12\) is at least 2. Thus we have at least two distinct choice of the good point \(a\).
				\end{proof}
				
				This improves our upper-bound of this case to \(\frac{68}{3}\) and thus we obtain $\sum_{|M| = m - 1, P}4h_0 + 3h_1 + 2h_2 + h_3\leq \frac{68}{3}(s+2)|\textbf{Ma}_{m-1}(P)|$ for point \(p\in P\) that sees a bad constellations.
			
			We look at the final case when points see same constellations and an edge is obtained ready for removal and reconstruction.
			
			\begin{thm}
				For any point \(p\in P\) that sees same constellation to the left and right as some other points, we obtain a unique correspondence between \(H_3(p,M) = 23\) or \(24\) and at least two distinct \(M'\) with \(H_3(p,M') \leq 22\)
			\end{thm}
				\begin{proof}
					The proof is straightforward. The edge immediately above (or below) \(p\) is ready for removal and reconstruction as implied by Lemma 2.2. We obtain \(M'\) by finding good points in the constellation that \(p\) sees and match them similar to the case in Theorem 3.3. 
				\end{proof}
			
			Again the upper-bound is improved to \(\frac{68}{3}\). Notice that the three cases sum up all possible situations for \(p\in P\) and therefore we know that 
			\begin{equation}
				\textbf{Ma}_{m}(P)\leq
				\frac{\frac{68}{3}(s+2)}{2n-6s}\textbf{Ma}_{m-1}(P).\\
			\end{equation}
		\subsection{\(H_4 = 5h_0 + 4h_1 + 3h_2 + 2h_3 + h_4\)}
			We similarly look at the number of good points in constellations that realizes specific values of \(L_4\) or \(R_4\).
			\begin{lem}
				Constellations that realize \(R_4\geq 21\) are the combination of two constellations, at least one of which can realizes \(R_3\geq 11\) below and above the bifurcation edge. In particular the combinations when \(R_4 = 21\) can be \(10 + 11\) or \(9 + 12\), which gives at least \(2\) good points in each case. It also holds for \(L_4\) and \(L_3\).
			\end{lem}
				\begin{proof}
					For \(R_4\geq 21\geq 18\), the bifurcation point is matched and thus \(r_0 = 0\). We obtain \(4r_1 + 3r_2 + 2r_3 + r_4\geq 21\) which is realized by two \(R_3\) constellations above and below bifurcation edge. Since each \(R_3\leq 12\) we readily obtain the result and thus we have at least \(2\) good points.. By symmetry the same argument applies to \(L_4\) and \(L_3\). 
				\end{proof}
			
			Thus we are now ready to derive improved bounds for \(H_4\) in the same way as we do to improve bounds for \(H_3\).
			
			\begin{thm}
				Consider any a set of \(k\) points \(\{p_1, p_2,..., p_k\}\), each having \(H_4(p_i, M) = x\) and separated by edges below and above, with all constellations being good and no two points seeing the same constellation. For \(x \geq 45\), we obtain at least \(k + 1\) good points that uniquely correspond to this setting. 
			\end{thm}
				\begin{proof}
					The proof is similar to that of Theorem 3.2. It is worth noticing that for \(x\geq 45\), we obtain \(L_4\geq 21\) and \(R_4\geq 21\). In either case, by Lemma 3.5 we have at least \(2\) good points in a constellation which can be charged at most twice, and since there are \(k+1\) distinct constellations, we obtain \(k + 1\) good points.
				\end{proof}
			
			Therefore the bound can be averaged to \(\frac{41 + 48}{2} = 44.5\). For points that see bad constellation, we have:
			\begin{thm}
				Consider any point \(p\in P\) that sees the bad constellation to the left or to the right. We have a unique correspondence between a constellation $M$ such that \(H_4(p,M) \geq 45\) and at least two distinct constellations \(M'\) with \(H_4(p,M') \leq 41\).
			\end{thm}
				\begin{proof}
					The proof is similar to that of Theorem 3.3. Notice that in this case we have at least four good points in a constellation and thus we obtain at least four distinct \(M'\) with \(H_4(p,M') \leq 41\).
				\end{proof}
			And the bound can be averaged to \(\frac{2*41 + 48}{3} = \frac{130}{3}<44.5\). Finally we derive the improved bound for the case when points see the same constellation.
			
			\begin{thm}
				For any point \(p\in P\) that sees same constellation to the left and right as some other points, we obtain a unique correspondence between a constellation $M$ such that \(H_4(p,M) \geq 45\) and at least two distinct constellations \(M'\) with \(H_4(p,M') \leq 41\)
			\end{thm}
				\begin{proof}
					The prove is same as the proof for Theorem 3.4.
				\end{proof}
			The bound is again being averaged to \(\frac{2*41 + 48}{3} = \frac{130}{3}<44.5\). Combining the three cases we readily obtain: 
			\begin{equation}
			\textbf{Ma}_{m}(P)\leq
			\frac{44.5(s+2)}{3n-7s}\textbf{Ma}_{m-1}(P).\\
			\end{equation}
\section{Reference}
[1]Micha Sharir and Emo Welzl, “On the number of crossing-free matchings, cycles, and partitions”, \textit{SIAM J. Comput.} 36 (2006), 695-720.\newline
\end{document}